\documentclass[11pt]{amsart}

\usepackage{amsmath, amsthm, amssymb}

\newcommand{\C}{\mathbb{C}}

\newcommand{\R}{\mathbb{R}}
\newcommand{\J}{\mathcal{J}}

\newcommand{\Ok}{\mathcal{O}}
\newcommand{\ordo}{\mathcal{O}}

\newcommand{\Pk}{{\mathbb{P}}}

\newcommand{\1}{{\bf 1}}
\newcommand{\w}{{\wedge}}
\newcommand{\codim}{{\text{codim}\,}}

\newtheorem{thm}{Theorem}[section]
\newtheorem{lma}[thm]{Lemma}

\newtheorem{prop}[thm]{Proposition}

\theoremstyle{definition}

\theoremstyle{remark}

\newtheorem{preremark}[thm]{Remark}
\newtheorem{preex}[thm]{Example}

\newenvironment{remark}{\begin{preremark}}{\qed\end{preremark}}
\newenvironment{ex}{\begin{preex}}{\qed\end{preex}}

\numberwithin{equation}{section}

\usepackage{xcolor}

\def\ep{\varepsilon}
\def\noqed{\renewcommand{\qedsymbol}{}}

\date{\today}

\begin{document}

\title[On a Monge-Amp\`ere Operator]
{On a Monge-Amp\`ere operator \\ for plurisubharmonic 
functions \\ with analytic singularities}

\author{Mats Andersson, Zbigniew B\l ocki, Elizabeth Wulcan}

\address{Department of Mathematics\\Chalmers University of Technology
and the University of
Gothenburg\\S-412 96 G\"OTEBORG\\SWEDEN}
\email{matsa@chalmers.se,  
wulcan@chalmers.se}
 \address{Institute of Mathematics\\ Jagiellonian University\\
\L ojasiewicza 6\\ 30-348 Krak\' ow \\ Poland}
\email{Zbigniew.Blocki@im.uj.edu.pl}

\begin{abstract}
We study continuity properties of generalized Monge-Amp\`ere operators for
plurisubharmonic functions with analytic singularities. In particular, 
we prove continuity for a natural class of decreasing approximating
sequences. 
We also prove a formula for the total mass of the Monge-Amp\`ere
measure of such a function on a compact K\"ahler mani\-fold. 
\end{abstract}

\thanks {The second author was supported by the Ideas Plus 
grant 0001/ID3/2014/63 of the Polish Ministry of Science and 
Higher Education.
The first and third author were  partially supported by the Swedish
  Research Council.}

\maketitle

\section{Introduction}
We say that a plurisubharmonic (psh) function $u$ on 
a complex manifold $X$ has {\it analytic singularities} 
if locally it can be written in the form 
\begin{equation}\label{as}
  u=c\log|F|+b,
\end{equation}
where $c\geq 0$ is a constant, $F=(f_1,\dots,f_m)$ is a tuple 
of holomorphic functions, and $b$ is bounded.
For instance, if $f_j$ are holomorphic functions and $a_j$ are positive
rational numbers, then
$\log(|f_1|^{a_1}+\cdots +|f_m|^{a_m})$ has analytic singularities.

By the classical Bedford-Taylor theory, \cite{BT1, BT2},  
if $u$ is of the form
\eqref{as}, then in $\{F\neq 0\}$, for any $k$,  
one can define a positive closed current $(dd^c u)^k$ recursively as 
\begin{equation}\label{snabel}
(dd^c u)^k:=dd^c\big (u(dd^c u)^{k-1}\big ).
\end{equation} 
It was shown in \cite{AW} that $(dd^cu)^{k}$ 
has locally 
finite mass near $\{F=0\}$ for any $k$ and that the 
natural extension
$\1_{\{F\neq 0\}}(dd^cu)^{k-1}$ across $\{F=0\}$ is closed, cf.\ 
\cite[Eq.\ (4.8)]{AW}. Moreover, by \cite[Proposition~4.1]{AW},
$u\1_{\{F\neq 0\}}(dd^cu)^{k-1}$ has locally finite mass as well, and therefore one can  define
the Monge-Amp\`ere current
\begin{equation}\label{def}
  (dd^cu)^k:=
dd^c\big(u\1_{\{F\neq 0\}}(dd^cu)^{k-1}\big)
\end{equation}
for any $k$.

Demailly, \cite{D3} extended Bedford-Taylor's definition 
\eqref{snabel}
to the case when the unbounded locus of $u$ is small compared to
$k$ in a certain sense; in particular, if $u$ is as in \eqref{as},
then $(dd^c u)^k$ is well-defined in this way as long as $k\le \codim\{F=0\}=:p$.  
Since, a positive closed current of bidegree $(k,k)$ with support on a
variety of codimension $>k$ vanishes, $\1_{\{F\neq
  0\}}(dd^cu)^k=(dd^cu)^k$ for $k\le p-1$, and it 
follows that \eqref{def} coincides with \eqref{snabel} 
for $k\leq p$.

Recall that the Monge-Amp\`ere operators $(dd^cu)^k$ defined by
Bedford-Taylor-Demailly have the following continuity property: 
if $u_j$ is a decreasing sequence of psh 
functions converging pointwise to $u$, then $(dd^cu_j)^k\to
(dd^cu)^k$ weakly. 
Moreover, a general psh function $u$ is said to be in the domain of 
the Monge-Amp\`ere operator $\mathcal D(X)$ 
if, in all open sets $\mathcal U\subset X$, 
$(dd^cu_j)^n$ converge to the same Radon measure for all decreasing
sequences of smooth psh $u_j$ converging to $u$ in $\mathcal U$. The
domain $\mathcal D(X)$ was characterized in 
\cite{Bma, Bajm}; in case $X$ is a hyperconvex domain in
$\C^n$ $\mathcal D(X)$ coincides with the Cegrell class, \cite{C}.

In this paper we study continuity properties of the Monge-Amp\`ere operators  
$(dd^c u)^k$ defined by \eqref{def}. 
It is not hard to see that  general psh functions with analytic 
singularities do not belong to $\mathcal D(X)$, cf.\ Examples 
~\ref{skola} and ~\ref{dagis}  
below, and therefore we do not have continuity
for all decreasing sequences in general. 
Our main result, however, states  
that continuity does hold for a large class of natural approximating
sequences. It thus provides an alternative definition of $(dd^cu)^k$, and 
at the same time gives further motivation for that 
this  Monge-Amp\`ere operator is indeed natural.

\begin{thm}\label{conv} 
Let $u$ be a negative psh function with analytic singularities on a
manifold of dimension $n$. 
Assume that $\chi_j(t)$ is a sequence of bounded nondecreasing
convex functions defined for $t\in(-\infty,0)$ decreasing
to $t$ as $j\to\infty$. Then for every $k=1,\dots,n$ we have
weak convergence of currents
$$
\big (dd^c(\chi_j\circ u)\big )^k\longrightarrow(dd^cu)^k
$$
as $j\to\infty$.
\end{thm}

For instance, we can take  $\chi_j=\max(t, -j)$ or $\chi_j=(1/2) \log( e^{2t}+1/j)$. 
Applied to $u=\log |F|$ and $\chi_j=(1/2)\log( e^{2t}+1/j)$ Theorem ~\ref{conv}
says that 
\[
\big(dd^c (1/2)\log (|F|^2+1/j)\big )^k \to 
 \big (dd^c\log |F|\big )^k,\]
 which was 
in fact proved already in \cite[Proposition~4.4]{A}.

By a resolution of singularities the proofs of 
various local properties of Monge-Amp\`ere
currents for psh functions with analytic singularities can be reduced
to the case of psh functions with \emph{divisorial singularities}, i.e., psh
functions that locally are of the form $u=c\log |f|+v$, where 
$c\geq 0$, $f$ is a
holomorphic function and $v$ is bounded. Since $\log|f|$ is
pluriharmonic on $\{f\neq 0\}$, in fact, $v$ is psh. 
In Section ~\ref{sec4} we prove Theorem~\ref{conv} for $u$ of this
form; in this case  
 \begin{equation}\label{skata0}
(dd^cu)^k=dd^c\big(u(dd^cv)^{k-1}\big)
    =dd^cu\wedge(dd^cv)^{k-1}.
\end{equation}
Note that, in light of the Poincar\'e-Lelong formula, 
\begin{equation*}
(dd^c u)^k =[f=0]\wedge (dd^cv)^{k-1}+(dd^c v)^k, 
\end{equation*}
where $[f=0]$ is the current of integration along $\{f=0\}$ counted
with multiplicities. 

\smallskip

Our definition of  $(dd^c u)^k$ thus relies on the possibility to reduce to the quite special case with
divisorial singularities. It seems like an extension to more general psh $u$ must involve some
further ideas, cf., Section~\ref{pastor}.

We also study psh functions with analytic singularities on 
compact K\"ahler manifolds. Recall that if $(X,\omega)$ 
is such a manifold then a function 
$\varphi\colon X\rightarrow\R\cup\{-\infty\}$ is called {\it
$\omega$-plurisubharmonic ($\omega$-psh)} if locally the function $g+\varphi$
is psh, where $g$ is a local potential for $\omega$, i.e.,  
$\omega=dd^cg$. Equivalently, one can require that 
$\omega+dd^c\varphi\geq 0$. We say that an $\omega$-psh
function $\varphi$ {\it has analytic singularities} if 
the functions $g+\varphi$ have analytic singularities. Note that such a
$\varphi$ is locally bounded outside an analytic variety $Z\subset X$
that we will refer to as the \emph{singular set} of $\varphi$. 
If $\varphi$ is an $\omega$-psh function with analytic singularities,
we can define a global positive current $(\omega + dd^c \varphi)^k$, by
locally defining it as $(dd^c (g+\varphi))^k$, see Lemma ~\ref{niko}.   
We will prove the following formula for the total 
Monge-Amp\`ere mass:

\begin{thm} \label{tmass}
Let $\varphi$ be an $\omega$-psh function with analytic
singularities on a compact K\"ahler manifold $(X,\omega)$ of dimension
$n$.  Let
$Z$ be the singular set of $\varphi$. Then
 \begin{equation}\label{north}
\int_X(\omega+dd^c\varphi)^n
      =\int_X\omega^n
      -\sum_{k=1}^{n-1}\int_X\1_Z(\omega+dd^c\varphi)^k
         \wedge\omega^{n-k}.
\end{equation}
\end{thm} 
In particular,
\begin{equation}\label{tvol}
   \int_X(\omega+dd^c\varphi)^n\leq\int_X\omega^n.
\end{equation}   

\begin{remark}
Let $\varphi$ be a general $\omega$-psh function such that the
Bedford-Taylor-Demailly Monge-Amp\`ere operator $(\omega +dd^c
\varphi)^n$ is well-defined; if $\varphi$ has analytic
singularities, this means that the singular set has dimension 0. 
Then it follows from Stokes' theorem that
equality holds in \eqref{tvol}. 
\end{remark}

To see that in general there is not equality in \eqref{tvol}
consider the following simple example:

\begin{ex}\label{ormar}
Let $X$ be the projective
space $\mathbb P^n$ with the Fubini-Study metric $\omega$ and let $n\ge 2$.
Define
  $$\varphi\big ([z_0:z_1:\ldots:z_n]\big ):=\log 
\Big (\frac{|z_1|}{|z|}\Big),\ \ \ 
     z\in\mathbb C^{n+1}\setminus\{0\}.$$ 
Since $(dd^c\log|z_1|)^n=0$ in $\mathbb C^{n+1}$, cf.\ \eqref{def}, it follows that
$(\omega+dd^c\varphi)^n=0$ 
on ~$\mathbb P^n$.
\end{ex}

In Section~\ref{kahl} we provide a geometric interpretation of Theorem~\ref{tmass}
which in particular
shows that inequality in \eqref{tvol} is not an  "exceptional case".

The paper is organized as follows. In Section~\ref{sec3} 
we prove a continuity result for currents of the form
  $$u\,dd^cv_1\wedge\dots\wedge dd^cv_k,$$
where $u$ is psh and $v_1,\dots,v_k$ are locally bounded psh, defined
by Demailly \cite{D}, cf.\ \eqref{skata0}.  %
In Section ~\ref{sec4} we prove Theorem~\ref{conv} for functions with 
divisorial singularities and we also characterize when such functions are maximal.
The general case of Theorem ~\ref{conv} is proved in
Section ~\ref{sec5}. In Section ~\ref{kahl} we prove 
Theorem ~\ref{tmass}. Finally in Section~\ref{pastor} we make some
further remarks.

Most of this work was carried out during the authors' visit at
the Centre for Advanced Study in Oslo and during the second 
named author's visit  in G\"oteborg.
We would like to thank Tristan Collins,
Eleonora Di Nezza,  S\l awomir Ko\l odziej, 
Duong Phong,  Alexander Rashkovskii, Valentino Tosatti, David Witt Nystr\"om, and Ahmed
Zeriahi 
for various discussions related to the subject of this paper.

We would also like to thank the referee for careful reading and important
comments on the first version.

\section{Continuity of certain Monge-Amp\`ere currents}\label{sec3} 

In the seminal paper \cite{BT2} Bedford and Taylor, see \cite[Theorem~2.1]{BT2},
showed that, for $k=1,\dots,n$ and locally bounded psh functions
$u,v_1,\dots,v_k$ on a manifold $X$ of dimension $n$, the current 
\begin{equation}\label{1}
u\,dd^cv_1\wedge\dots\wedge dd^cv_k
\end{equation}
is well-defined and continuous for decreasing sequences. 
Demailly generalized their definition to the case when $u$ is 
merely psh; he proved that the current 
\eqref{1} has locally finite mass, see  \cite[Theorem~1.8]{D}. 
Here we prove the corresponding continuity result. 

\begin{thm}\label{appr}
Assume that $u^j$ is a sequence of psh functions decreasing
to a psh function $u$ and that for $\ell=1,\dots,k$ the 
sequence $v^j_\ell$ of psh functions decreases to a locally bounded
psh $v_\ell$ as $j\to\infty$. Then 
  $$u^j\,dd^cv^j_1\wedge\dots\wedge dd^cv^j_k
    \longrightarrow u\,dd^cv_1\wedge\dots\wedge dd^cv_k$$
weakly as $j\to\infty$.
\end{thm}

\begin{proof}\noqed
By the Bedford-Taylor theorem we have weak convergence
  $$S^j:=dd^cv_1^j\wedge\dots\wedge dd^cv_k^j
     \longrightarrow dd^cv_1\wedge\dots\wedge dd^cv_k=:S.$$
By \cite[Theorem~1.8]{D} the sequence $u^jS^j$ is locally 
weakly bounded and thus it is enough to show that, if 
$u^jS^j\to\Theta$ weakly, then $\Theta=uS$.

Take an elementary positive form $\alpha$ of bidegree $(n-k,n-k)$
and fix $j_0$ and $\ep>0$. Then for $j\geq j_0$ we have
  $$u^jS^j\wedge\alpha\leq u^{j_0}S^j\wedge\alpha
     \leq u^{j_0}\ast\rho_\ep\,S^j\wedge\alpha,$$
where $u^{j_0}\ast\rho_\ep$ is a standard regularization of $u^{j_0}$ 
by convolution, i.e., $\rho_\epsilon$ is a rotation
invariant approximate indentity.
Letting $j\to\infty$ we get 
$\Theta\wedge\alpha\leq u^{j_0}\ast\rho_\ep\,S\wedge\alpha$ 
and thus $\Theta\leq uS$. 

We will use the following lemma.
\end{proof}

\begin{lma}\label{int-par}
Let $u,v_0,v_1,\dots,v_n$ be psh functions
defined in a neighborhood of $\overline\Omega$ where $\Omega$ 
is a bounded domain in $\mathbb C^n$. Suppose that  all of these 
functions except possibly $u$ are bounded and set 
$T:=dd^cv_2\wedge\dots\wedge dd^cv_n$. Assume that 
$v_0\leq v_1$ in $\Omega$ and $v_0=v_1$ in $\Omega\cap U$,  
where $U$ is a neighborhood of $\partial\Omega$. Then
  $$\int_\Omega u\,dd^cv_0\wedge T\leq 
     \int_\Omega u\,dd^cv_1\wedge T.$$
\end{lma}

\begin{proof}
We have 
\begin{align*}
\int_\Omega u\,dd^c v_0\wedge T-\int_\Omega u\,dd^cv_1\w T
  &=\lim_{\ep\to 0}\int_\Omega u\ast\rho_\ep\,dd^c(v_0-v_1)\wedge T\\
  &=\lim_{\ep\to 0}\lim_{\delta\to 0}\int_\Omega u\ast\rho_\ep\, 
       dd^c\big ((v_0-v_1)\ast\rho_\delta\big )\wedge T\\
  &=\lim_{\ep\to 0}\lim_{\delta\to 0}
      \int_\Omega(v_0-v_1)\ast\rho_\delta\,dd^c(u\ast\rho_\ep) 
          \wedge T \leq 0. 
\end{align*}
\end{proof}

\begin{proof}[End of proof of Theorem~\ref{appr}]
We may assume that all functions are defined in a neighborhood
of a ball $\overline B=\overline B(z_0,r)$ and, similarly as in the proof of
Bedford-Taylor's theorem,
that  $v_\ell^j=v_\ell=A(|z-z_0|^2-r^2)$ near 
$\partial B$ for some $A>0$, cf.,  e.g., the proof of \cite[Theorem~1.5]{D}.
 Since $\Theta\leq uS$, it remains 
to prove that $\int_B(uS-\Theta)\wedge\omega^{n-k}\leq 0$, where 
$\omega=dd^c|z|^2$. By successive application of Lemma ~\ref{int-par}
we get
  $$\int_Bu\,dd^cv_1\wedge\dots\wedge dd^cv_k\wedge\omega^{n-k}
     \leq\int_Bu^jdd^cv_1^j\wedge\dots\wedge dd^cv_k^j\wedge
        \omega^{n-k}.$$
Therefore,
  $$
\int_Bu\,S\wedge\omega^{n-k}\leq 
\liminf_{j\to\infty}  \int_Bu^jdd^cv_1^j\wedge\dots\wedge
dd^cv_k^j\wedge\omega^{n-k}\leq 
\int_B\Theta\wedge\omega^{n-k},$$
and thus the theorem follows.
\end{proof}

Theorem \ref{appr} generalizes a result of Demailly (see 
\cite{Dembook}, Proposition III.4.9 on p.\ 155) who assumed 
in addition that a complement of the open set where 
$u,v_1,\ldots,v_k$ are locally bounded has vanishing
$(2n-1)$-dimensional Hausdorff measure.

\section{The case of divisorial singularities}\label{sec4}
In this section we first prove a special case of Theorem ~\ref{conv}. 

\begin{thm} \label{div}
Assume that $u=\log|f|+v$ is negative, where $f$ is 
holomorphic and $v$ is a bounded psh function. Let $\chi_j$ be as 
in Theorem ~\ref{conv}. Then
  $$\big(dd^c(\chi_j\circ u)\big)^k\longrightarrow dd^cu\wedge(dd^cv)^{k-1}$$
as $j\to\infty$.
\end{thm}


\begin{proof} We will use an idea from \cite{Bproc}.
Notice that locally on $(-\infty,0)$, the sequence $\chi'_j$ 
is bounded and tends to $1$ uniformly when $j\to \infty$.
For each  $j$, 
$$
\gamma_j(t):=\int_{-1}^t (\chi_j'(s))^k ds+ \chi_j(-1)
$$
is bounded, convex and  nondecreasing on $(-\infty,0)$,  and
$\gamma_j'=(\chi_j')^k$, where the derivative exists.  Moreover,
the sequence $\gamma_j$ is decreasing and tends to $t$.

Let us first assume that $\chi_j$, and hence $\gamma_j$, are smooth.
Since $\log|f|$ is pluriharmonic on $\{f\neq 0\}$  we have that
\begin{align*}%
\big (dd^c(\chi_j\circ u)\big )^k
   &=\big (\chi_j''\circ u\,du\wedge d^cu+\chi_j'\circ u\,dd^cu\big )^k\\
   &=\big (k\chi_j''\circ u\,du\wedge d^cu+\chi_j'\circ u\,dd^cu\big )
      \wedge\big (\chi'\circ u \, dd^cu\big )^{k-1}\\
   &=d\big((\chi_j'\circ u)^kd^cu\big)\wedge(dd^cu)^{k-1}\\
   &=dd^c(\gamma_j\circ u)\wedge(dd^cv)^{k-1}\\
   &=dd^c\big(\gamma_j\circ u\,(dd^cv)^{k-1}\big)
\end{align*}
there. Since none of the above currents charges the set $\{f=0\}$,
the equality 
\begin{equation}\label{orm}
(dd^c(\chi_j\circ u))^k=dd^c\big(\gamma_j\circ u\,(dd^cv)^{k-1}\big)
\end{equation}
holds everywhere. If $\chi_j$ is not smooth we make
a regularization $\chi_{j,\epsilon}=\chi_j* \rho_\epsilon$. Then $\chi_{j,\epsilon}'\to
\chi_j'$ in $L^1_{loc}(-\infty,0)$ and hence the associated $\gamma_{j,\epsilon}$ tend to $\gamma_j$ locally uniformly. 
We conclude that \eqref{orm} still holds. 
The theorem now follows from \eqref{orm} and Theorem ~\ref{appr}.
\end{proof}

The following example shows that $(dd^c u_j)^k$ does not converge to
$(dd^c u)^k$ for general decreasing sequences of psh functions $u_j\to
u$. 

\begin{ex}\label{skola}
 Let
\begin{equation*}
  u(z)=\log|z_1|+|z_2|^2.
\end{equation*}
One easily checks that 
 $$
(dd^cu)^2=[z_1=0]\wedge dd^c|z_2|^2\neq 0.
$$
Thus, if $u_j=\chi_j\circ u$, where $\chi_j$ is chosen as
Theorem~\ref{conv}, e.g.,  
$u_j=(1/2)\log(|z_1|^2e^{2|z_2|^2}+1/j)$, then 
$$
(dd^cu_j )^2\to (dd^cu)^2.
$$
However,  
$v_j:=(1/2)\log(|z_1|^2+1/j)+ |z_2|^2$
are also smooth psh functions that 
decrease to $u$ but 
$$
(dd^cv_j )^2\to 2[z_1=0]\w dd^c|z_2|^2=2(dd^c u)^2.
$$

It follows that $u$ does not 
belong to the domain of definition of the Monge-Amp\`ere operator; in
fact, this follows directly from \cite[Theorem~1.1]{Bma} since clearly
$u\notin W^{1,2}_{loc}$. 
By \cite[Theorem~4.1]{Bma} one can find another 
approximating sequence of smooth psh functions decreasing to $u$ 
whose Monge-Amp\`ere measures do not have locally uniformly finite
mass near $\{z_1=0\}$. 
\end{ex}

Recall that a psh function $u$ is called 
{\it maximal} in an open set $\Omega$ in $\C^n$ if for any other psh $v$ in $\Omega$
satisfying $v\le u$ outside a compact set, we have $v\le u$ in $\Omega$.
 We refer to \cite{S, Bln} for 
basic properties of maximal psh functions.  In particular, $u$ is maximal if and only if
for each $\Omega'\Subset \Omega$ and psh $v$ such that $v\le u$ on
$\partial\Omega'$ one has $v\le u$ in $\Omega'$. 
By Bedford-Taylor's 
theory \cite{BT1, BT2} a locally bounded psh $u$ 
is maximal if and only if 
$(dd^cu)^n=0$. 

The following result due to Rashkovsii, see \cite[Theorem~1]{R}, gives a local characterization of 
maximal psh functions with divisorial 
singularities. 

\begin{prop}\label{skata} 
Let $\Omega$ be a domain in $\mathbb C^n$, $n\ge 2$,  $f$ 
a holomorphic function in $\Omega$ (not vanishing
identically),  and $v$ a locally bounded psh function in $\Omega$. Then 
$u=\log|f|+v$ is maximal in $\Omega$ if and only if $v$ is maximal
in $\Omega$.
\end{prop}

One can rephrase Proposition 3.3 as follows: if a psh function
u is globally of the form $\log|f|+v$, where f is a holomorphic
function and v is psh and locally bounded, then u is maximal if
and only if it is maximal outside the singular set. It would be
interesting to verify whether such a characterization is true
globally for psh functions with divisorial singularities.

\begin{ex}\label{dagis}
Proposition ~\ref{skata} implies that the psh function $u$ in Example
~\ref{skola} is maximal (in any domain in $\C^2$).
Thus it is not true in general for psh functions with analytic
singularities $u$ that $(dd^c u)^n=0$ is equivalent to $u$ being maximal.

Moreover in any bounded domain 
we can find a sequence of continuous maximal psh 
functions decreasing to $u$, or a sequence $u_j$ of smooth psh functions 
decreasing to $u$ such that $(dd^c u_j)^2\to 0$ weakly, 
see e.g., \cite[Proposition~1.4.9]{Bln}. 
It follows that (the mass of) $\lim_j(dd^c u_j)^2$ when $u_j$ is a decreasing
sequence of bounded psh functions $u_j\to u$ can be
both smaller and larger than (the mass of) $(dd^cu)^2$, cf.\ Example
~\ref{skola}. 
\end{ex}
 
\begin{remark}
In \cite{Buzb} it was shown that the psh function 
\begin{equation}\label{uzb}
u(z):=-\sqrt{\log|z_1|\log|z_2|}
\end{equation}
is maximal in $\{|z_1|<1,\ |z_2|<1\}\setminus\{(0,0)\}$, and that the
Monge-Amp\`ere measure of $\max\{u,-j\}$, however, does not converge weakly
to 0 as $j\to\infty$. 

In view of Theorem ~\ref{div} and Proposition ~\ref{skata} the function
$u$ in Examples ~\ref{skola} and \ref{dagis} 
gives a new example of such a maximal psh function. 
\end{remark}

Proposition~\ref{skata} implies that for psh functions with divisorial 
singularities it suffices to check their maximality outside
hypersurfaces. This 
is not true in general 
as the following
example shows. 

\begin{ex}
The function given by \eqref{uzb} is psh in the unit bidisc,
maximal away from the singular set, i.e. the hypersurface $\{z_1z_2=0\}$, but not maximal in the
entire bidisc $\Delta^2$. In fact, the psh function 
  $$v(z):=-\sqrt{-\log|z_1|}-\sqrt{-\log|z_2|}+1$$
coincides with $u$  
on the boundary of the bidisk
$(\Delta(0,1/e))^2$, but $v>u$ on the diagonal inside $(\Delta(0,1/e))^2$.
\end{ex}

\section{The general case of Theorem~\ref{conv}}\label{sec5} 

We now give a proof of Theorem ~\ref{conv}. 
Since the statement is local we may assume that $u=\log|F|+b$,
where $F$ is a tuple of holomorphic functions on an open set 
$X\subset\C^n$, and $b$ is bounded.   

Let $Z$ be the common zero set of $F$.
By Hironaka's theorem one can find a proper map
$\pi\colon X'\to X$ that is a biholomorphism 
$X'\setminus\pi^{-1}Z\simeq X\setminus Z$, where $\pi^{-1}Z$ 
is a hypersurface,  such that the ideal sheaf generated by the functions
$\pi^*f_j$ is principal.  Let $D$ be the exceptional divisor and let $L\to X'$ be 
the associated line bundle that has a global holomorphic section $f^0$ whose
divisor is precisely $D$. It then follows that
$\pi^\ast F=f^0 F'$, where 
 $F'$ is a nonvanishing tuple of sections of $L^{-1}$. Given a local frame for 
$L$ on $X'$ we can thus write $F=f^0F'$ where $f^0$ is a holomorphic
function and $F'$  a nonvanishing tuple of holomorphic functions. 
Then
\begin{equation*}
\pi^\ast u=\log|\pi^\ast F|+\pi^\ast b=
\log|f^0|+\log|F'|+\pi^\ast b=:\log|f^0|+v,
\end{equation*}
and since $\pi^\ast u$ is psh it follows that $v$ is.  Another local
frame gives rise to the same local decomposition up to a pluriharmonic
function.  Notice that
$$
dd^c\log|f^0|=[D],
$$
where $D$ is the divisor determined by $f^0$.

In view of Theorem ~\ref{div}, 
\begin{equation*}
\big(dd^c(\chi_j\circ \pi^*u)\big)^k\to (dd^c\pi^*u)^k=[D]\w (dd^c v)^{k-1}+(dd^c v)^k.
\end{equation*}
Assume that $a$ is psh and bounded. Since neither $(dd^c a)^k$ nor
$(dd^c \pi^*a)^k$ charge subvarieties it follows that 
\begin{equation*}
\pi_* (dd^c \pi^*a)^k=(dd^c a)^k.
\end{equation*}
Since 
$
\pi^*(\chi_j\circ u)=\chi_j\circ \pi^*u, 
$
thus
$$
\big (dd^c(\chi_j\circ u)\big)^k
=
\pi_*\big(dd^c\big(\pi^*(\chi_j\circ u)\big)\big)^k
=
\pi_*\big(dd^c(\chi_j\circ\pi^* u)\big)^k\to
\pi_*\big([D]\w (dd^c v)^{k-1}+(dd^c v)^k\big).
$$
By \cite[Equation~(4.5)]{AW}, 
$$  
\pi_*\big([D]\w (dd^c v)^{k-1}+(dd^c v)^k\big)=(dd^c u)^k
$$
and thus  Theorem ~\ref{conv} follows.

\begin{remark}\label{nons}
The definition of $(dd^c u)^k$  as 
well as proof of Theorem~\ref{conv} 
work just as well if $X$ is a reduced, not necessarily smooth, analytic space,
cf., e.g., \cite{ASWY}.
\end{remark}

\section{Proof and discussion of Theorem ~\ref{tmass}}\label{kahl}
We start by showing that the Monge-Amp\`ere operators 
$(\omega + dd^c \varphi)^k$ are well-defined whenever $\varphi$ is an
$\omega$-psh function with analytic singularities. 

\begin{lma}\label{niko} 
Let $\varphi$ be an $\omega$-psh function with analytic singularities. Then $(dd^c (g+\varphi))^k$
is independent of the local potential $g$ of $\omega$. 
\end{lma}

\begin{proof}
We need to prove that 
\begin{equation}\label{solros}
\big (dd^c (g+h+\varphi)\big)^k = \big (dd^c (g+\varphi)\big )^k
\end{equation}
if $h$ is pluriharmonic. 
Clearly this is true for $k=1$. 

If $T$ is a positive closed current and $u$ and $v$ are functions such
that $uT$ and $vT$ have locally finite mass, then clearly so has
$(u+v)T=uT+vT$. 
Assuming that \eqref{solros} holds for $k=\ell$, it follows that 
\begin{multline*}
\big (dd^c
(g+h+\varphi)\big)^{\ell+1}=
dd^c\big((g+h+\varphi)\1_{X\setminus
  Z}\big (dd^c (g+h+\varphi)\big)^\ell\big)=\\
dd^c\big((g+\varphi)\1_{X\setminus Z}\big (dd^c (g+\varphi)\big)^\ell\big)+
dd^c\big (h\1_{X\setminus Z}\big (dd^c (g+\varphi)\big)^\ell\big), 
\end{multline*}
where $Z$ is the singular set of $\varphi+g$. 
Since $h$ is pluriharmonic the rightmost expression equals
\[
\big (dd^c (g+\varphi)\big )^{\ell+1} + dd^ch \wedge \1_{X\setminus Z}
(dd^c(g+\varphi)\big )^\ell=\big (dd^c (g+\varphi)\big )^{\ell+1}.
\]
Thus \eqref{solros} follows by induction. 
\end{proof}

\begin{proof}[Proof of Theorem~\ref{tmass}] 
For $k=0,\dots,n-1$ we let 
  $$T_k:=\1_{X\setminus Z}(\omega+dd^c\varphi)^k;$$
note that $T_0$ is just the function 1. 
Locally we can define
\begin{equation}\label{east}
\varphi T_k:= (g+\varphi)T_k - gT_k, 
\end{equation}
cf.\ \eqref{def}. 
This definition is independent of the local potential $g$ of $\omega$
and, cf.\ the proof of Lemma \ref{niko}, 
thus $\varphi T_k$ defines a global current on $X$. 
Applying $dd^c$ to \eqref{east} we get 
\begin{equation}\label{west}
dd^c (\varphi T_k)=dd^c\big ((g+\varphi)T_k\big) - dd^c(gT_k)= (\omega
+dd^c\varphi)^{k+1}-\omega \wedge T_k. 
\end{equation}
Now
\begin{multline}\label{south}
\int_X\omega^{n-k}\wedge T_k
=
\int_X\omega^{n-k-1}\wedge (\omega+dd^c\varphi)^{k+1}- 
\int_X\omega^{n-k-1}\wedge dd^c (\varphi T_k)
=\\
\int_X\omega^{n-k-1}\wedge \1_{Z}(\omega+dd^c\varphi)^{k+1}+ 
\int_X\omega^{n-k-1}\wedge T_{k+1}. 
\end{multline}
Here we have used \eqref{west} for the second equality; the second
term in the middle expression vanishes by Stokes' theorem. 
Applying \eqref{south} inductively to
$\int_X\omega^n=\int_X\omega^nT_0$ we get \eqref{north}. 
\end{proof}

Given an $\omega$-psh function $\varphi$, in \cite{GZ,BEGZ} was
introduced the \emph{non-pluripolar Monge-Amp\`ere operators} 
\[
\big\langle(\omega+dd^c\varphi)^k\big\rangle:= 
\lim_{j\to\infty}\1_{\{\varphi>-j\}}\big(\omega + dd^c\max(\varphi,
-j)\big)^k; 
\]
the definition is based on the corresponding local consctruction 
in \cite{BT3}.

Assume that $\varphi$ has analytic singularities with singular set
$Z$. Then 
$\langle(\omega+dd^c\varphi)^k\rangle$ coincides with the classical 
Monge-Amp\`ere operator outside $Z$ and it does not charge $Z$. Hence 
\[
\big\langle(\omega+dd^c\varphi)^k\big\rangle = \1_{X\setminus Z}
(\omega + dd^c\varphi)^k.
\]
Following \cite{AW}, cf.\ \cite{ASWY}, we let 
$$
M_k^\varphi:=\1_Z(dd^c\varphi+\omega)^k, \ k=1, \ldots, n.
$$
Using this notation we can rephrase Theorem~\ref{tmass}  as
\begin{equation}\label{pluto}
\int_X \big \langle (\omega+dd^c\varphi)^n \big \rangle 
=\int_X \omega^n-\sum_{k=1}^n\int_X M_k^\varphi\w\omega^{n-k}.
\end{equation}
In fact, by applying \eqref{south} inductively to $\int_X\omega^nT_0$
as in the proof of Theorem ~\ref{tmass}, but
stopping at $k=\ell-1$, we get: 
\begin{prop}\label{milk}
Let $\varphi$ be an $\omega$-psh function with analytic
singularities on a compact K\"ahler manifold $(X,\omega)$ of dimension
$n$. 
Then, for $\ell=1,\ldots, n$,  
\begin{equation}\label{juvel}
\int_X \big \langle (\omega+dd^c\varphi)^\ell \big \rangle \wedge
\omega^{n-\ell} 
=\int_X \omega^n -\sum_{k=1}^\ell\int_X M_k^\varphi\w\omega^{n-k}.
\end{equation}
\end{prop}

From \cite[Theorem~1.16]{BEGZ} it follows that if $\varphi, \varphi'$
are $\omega$-psh with analytic singularities and $\varphi$ is
\emph{less singular} than $\varphi'$, i.e., $\varphi\geq
\varphi'+\ordo (1)$, then 
\begin{equation}\label{nonpolar}
\int_X \big\langle(\omega+dd^c\varphi)^\ell\big\rangle \wedge \omega^{n-\ell}\geq 
\int_X \big\langle(\omega+dd^c\varphi')^\ell\big\rangle \wedge \omega^{n-\ell}
\end{equation}
for each $\ell$. 
From \eqref{nonpolar} and Proposition ~\ref{milk} we conclude that
\[ 
\sum_{k=1}^\ell\int_X M_k^\varphi\w\omega^{n-k} \leq \sum_{k=1}^\ell\int_X M_k^{\varphi'}\w\omega^{n-k} 
\]
for each $\ell$. It is not true in general, however, that $\int_X
M_k^\varphi\w\omega^{n-k} \leq \int_X M_k^{\varphi'}\w\omega^{n-k}$
for each $k$, as is illustrated by the following example. 

\begin{ex}\label{segre}
Let $X=\Pk^2_{[z_0:z_1:z_2]}$ with the Fubini-Study metric $\omega$, and let
\[
\varphi=\log\Big (\frac{(|z_1|^2+|z_2|^2)^{1/2}}{|z|}\Big ) \text{
  and } 
\varphi'=\log\Big (\frac{|z_1|}{|z|}\Big ), 
\] 
cf.\ Example ~\ref{ormar}. 
Then $\varphi$ and $\varphi'$ are $\omega$-psh with analytic
singularities and clearly $\varphi$ is
less singular than $\varphi'$. 
Note that $M_2^\varphi=[z_1=z_2=0]$ and $M_1^{\varphi'}=[z_1=0]$,
whereas $M_1^\varphi$ and $M_2^{\varphi'}$ vanish. In particular,  
$\int_X M_2^\varphi>\int_X M_2^{\varphi'}$. 
\end{ex}

\begin{remark}
In general we cannot have a global continuity result like Theorem
~\ref{conv}. Indeed, assume that $\varphi$ is an $\omega$-psh function
with analytic singularities such that 
\[
\int_X(\omega+dd^c\varphi)^\ell\wedge \omega^{n-\ell}
      <\int_X\omega^n,\]
 cf.~ \eqref{juvel}; 
this holds, e.g., for $\varphi'$ in
Example ~\ref{segre} and $\ell=2$. 
Moreover, assume that there is a sequence of locally bounded
$\omega$-psh, or smooth, functions $\varphi_j$ converging to
$\varphi$. By Stokes' theorem 
\[
\int_X(\omega+dd^c\varphi_j)^\ell\wedge \omega^{n-\ell}
      =\int_X\omega^n
\]
for all $j$, and thus 
%
 $(\omega+dd^c\varphi_j)^\ell$ cannot converge to
$(\omega+dd^c\varphi)^\ell$. 

\end{remark}

Let $X$ be a, possibly non-smooth, analytic space,  cf. Remark~\ref{nons}, and let
$\omega$ be a smooth positive $(1,1)$-form on $X$ that locally has a smooth
potential.  Then we still have the notion of $\omega$-psh function on $X$ and
the formulation and proof of Theorem~\ref{tmass}, as well as
the definitions of $M_k^\varphi$,  work as in the smooth case.

There is a close connection between Theorem ~\ref{tmass} and the
currents $M_k^\varphi$ and 
global (nonproper)
intersection theory, that will be studied in
a forthcoming paper by two of the authors. In some sense the currents
$M_k^\varphi$ can be seen as generalized intersection cycles, cf.\ \cite[Section~6]{ASWY}. Let us just give a simple example
with a proper intersection here, cf.\ Example~\ref{ormar} above. 

\begin{ex}  
Let $i\colon X\to \Pk^n$ be a projective variety of
dimension $p$, and
let $f$ be a $m$-homogeneous form in $\C^{n+1}$ that does not vanish identically on any
irreducible component of $X$; i.e., $Z(f)$ intersects $X$ properly.  If we consider
$f$ as a section of the line bundle $\Ok(m)\to \Pk^n$ then it has the natural norm
$\|f\|=|f(z)|/|z|^m$.  It follows that $u=\log\|f\|$ is $m\omega$-psh
on $X$, where $\omega$ is the Fubiny-Study form.
Notice that $\langle (m\omega + dd^c \varphi)^n\rangle=0$. Moreover $M_k=0$ for $k\ge 2$ and $M_1=dd^c\log|f|$. Thus
the equality \eqref{pluto} means that
$$
\int_X dd^c\log|f|\w\omega^{p-1}=m\int_X \omega^p=\deg Z\cdot \deg X
$$
and the rightmost expression is equal to 
\begin{equation}\label{sudda}
\int_{\Pk^n} [Z]\wedge [X]\w \omega^{p-1}.
\end{equation}
Since $[Z]\wedge [X]$ is the Lelong 
current of the proper intersection $Z\cdot X$ of $Z$ and $X$,
\eqref{sudda} equals $\deg (Z\cdot X)$ and thus 
\eqref{pluto} in this case is just an instance of  Bezout's formula.  
\end{ex}

\section{Some further comments}\label{pastor}

The Monge-Amp\`ere operators \eqref{def} are also closely related to
local intersection theory. 
Given a psh function of the form \eqref{as} on a possibly non-smooth 
analytic space $X$, we let 
$$
M^u_k:=\1_Z(dd^c u)^k, \ k=1, \ldots, n, 
$$
where $Z=\{F=0\}$.  
In \cite{AW, ASWY} it was proved that 
\begin{equation}\label{whitney}
\ell_x M^u_k = e_k(x), 
\end{equation}
where $\ell_x\mu$ denotes the Lelong number of the positive closed current $\mu$ at
$x$, and $e_k(x)$ is the $k$th \emph{Segre number} at $x$ of the ideal $\J$
generated by $F$. 
Segre numbers were introduced independently by Gaffney-Gassler
\cite{GG} and Tworzewski \cite{T} as certain local intersection
numbers, and in a purely algebraic way by Achilles-Manaresi
\cite{AM}. In fact, if $Z$ is discrete, then the only nonvanishing
Segre number $e_n(x)$ equals the classical \emph{Hilbert-Samuel
  multiplicity} of $\J$ at $x$. Thus \eqref{whitney} is a
generalization of the well-known fact the Lelong number of
$(dd^c\log|F|)^n$ is the Hilbert-Samuel multiplicity of $\J$ if $Z$ is
discrete.

\smallskip 

Demailly's approximation theorem \cite{D2} asserts that any
psh function $u$ on a bounded pseudoconvex domain $\Omega$ can be approximated by psh functions with analytic
singularities. 
Let   
$$u_j:=\frac 1{2j}\log\sup\left\{|f|^2\colon f\in\mathcal 
      O(\Omega),\ \int_\Omega|f|^2e^{-2ju}d\lambda\leq 1
        \right\}.$$
Then $u_j\to u$ pointwise and in $L^1_{loc}$ and 
there exists a sequence of positive constants $\ep_j$
decreasing to 0 such that the subsequence $u_{2^j}+\ep_j$
is decreasing, see \cite{DPS}; in view of \cite{K} this cannot be done
for the whole sequence $u_j$. Since $u_j$ are in fact defined
by weighted Bergman kernels, it is clear that locally they
can be written in the form \eqref{as} where $b$ is smooth.
If $u$ has an isolated analytic singularity
(so that the Demailly definition of the Monge-Amp\`ere operator applies),
it is proved in \cite{R1} that there is continuity for the Monge-Amp\`ere
masses of the $u_j$.
It would be interesting to investigate possible convergence properties
of  $(dd^c u_j)^k$  
in more general cases; for example when the initial function
$u$ also has analytic singularities,  or for more general
psh $u$ as a means to extend $(dd^c u)^k$  to such  $u$.


\begin{thebibliography}{10}


\bibitem{AM}{\sc R.\ Achilles, M.\ Manaresi},
{\sl Multiplicities of a bigraded ring and intersection theory}
Math.\ Ann.\ 309 (1997) 573--591.

\bibitem{A} {\sc M.\ Andersson}, {\sl Residues of holomorphic
sections and Lelong currents}, Ark.\ Mat.\ 43 (2005), 201--219.

\bibitem{AW} {\sc M.\ Andersson, E.\ Wulcan}, {\sl Green functions, 
Segre numbers, and King's formula}, Ann.\ Inst.\ Fourier 64 (2014), 
2639--2657.

\bibitem{ASWY} {\sc M.\ Andersson, H.\ Samuelsson Kalm, E.\ Wulcan,
A.\ Yger} {\sl Segre numbers, a generalized King formula, and local intersections}
J.\ Reine und Angewandte Math.\  728 (2017), 105--136.

\bibitem{BT1} {\sc E.\ Bedford, B.A.\ Taylor}, {\sl The Dirichlet problem
for a complex Monge-Amp\`ere equation}, Invent.\ Math.\ 37 (1976),
1--44.

\bibitem{BT2} {\sc E.\ Bedford, B.A.\ Taylor}, {\sl A new capacity for
plurisubharmonic functions}, Acta Math.\ 149 (1982), 1--41.

\bibitem{BT3} {\sc E.\ Bedford, B.A.\ Taylor}, 
{\sl Fine topology, \v Silov boundary, and $(dd^c)^n$},  
J.\ Funct.\ Anal.\ 72 (1987), no. 2, 225--251.

\bibitem{Bproc}{\sc Z.\ B\l ocki}, {\sl Equilibrium measure 
of a product subset of $\mathbb C^n$}, Proc.\ Amer.\ Math.\ Soc.\
128 (2000), 3595--3599.

\bibitem{Bln} {\sc Z.\ B\l ocki}, {\sl The complex Monge-Amp\`ere 
operator in pluripotential theory}, lecture notes, 2002, available at
{\tt http://gamma.im.uj.edu.pl/$\widetilde{\phantom{p}}$blocki}.

\bibitem{Bma} {\sc Z.\ B\l ocki}, {\sl On the definition of the
Monge-Amp\`ere operator in $\mathbb C^2$}, Math.\ Ann.\ 328
(2004), 415--423.

\bibitem{Bajm} {\sc Z.\ B\l ocki}, {\sl The domain of definition 
of the complex Monge-Amp\`ere operator}, Amer.\ J.\ Math.\ 128
(2006), 519--530.

\bibitem{Buzb} {\sc Z.\ B\l ocki}, {\sl A note on maximal
plurisubharmonic functions}, Uzbek Math.\ J.\ 1 (2009), 28--32.

\bibitem{BEGZ}{\sc  S.\ Boucksom, P.\ Eyssidieux, V.\ Guedj and A.\ Zeriahi},
{\sl  Monge-Ampe`re equations in big cohomology classes}
Acta Math.\  {205},  (2010),  199--262.

\bibitem{C} {\sc U.\ Cegrell}, 
{\sl The general definition of the complex Monge-Amp\`ere operator.}
Ann.\ Inst.\ Fourier (Grenoble) 54 (2004), no. 1, 159-–179.

\bibitem{D} {\sc J.-P.\ Demailly}, {\sl Potential theory in several
complex variables}, lecture notes, 1989, available at 
{\tt http://www-fourier.ujf-grenoble.fr/$\widetilde{\phantom{p}}$demailly}.

\bibitem{D2} {\sc J.-P.\ Demailly}, {\sl  Regularization
of closed positive currents and intersection theory}, J.\ Alg.\
Geom.\ 1 (1992), 361--409.

\bibitem{D3}{\sc J.-P.\ Demailly}, {\sl  Monge-Amp\`ere operators, Lelong numbers and intersection theory}, Complex analysis and geometry, 115–193, Univ.\ Ser.\ Math.\, Plenum, New York, 1993. 

\bibitem{Dembook} {\sc J.-P. Demailly}, {\sl Complex Analytic and
Differential Geometry}, 1997, available at {\tt
http://www-fourier.ujf-grenoble.fr/$\widetilde
{\phantom{a}}$demailly/books.html}

\bibitem{DPS} {\sc J.-P.\ Demailly, T.\ Peternell, M.\ Schneider},
{\sl Pseudo-effective line bundles on compact K\"ahler manifolds},
Internat.\ J.\ Math.\ 12 (2001), 689--741.

\bibitem{GG}{\sc T.\ Gaffney, R.\ Gassler},
{\sl Segre numbers and hypersurface singularities} 
J.\ Algebraic Geom.\ {8} (1999), 695--736.

\bibitem{GZ}{\sc V.\ Guedj and A.\ Zeriahi}
{\sl  The weighted Monge-Ampe`re energy of quasiplurisubharmonic
functions} 
J.\ Funct.\ Anal.\\  {15},  (2005),  607--639.

\bibitem{K} {\sc D.\ Kim}, {\sl A remark on the approximation of 
plurisubharmonic functions}, C.\ R.\ Math.\ Acad.\ Sci.\ Paris 
352 (2014), 387--389.

\bibitem{R}{\sc A.\ Rashkovskii}, {\sl Maximal plurisubharmonic
    functions associated with holomorphic mappings}, Indiana Univ.\
  Math.\ J. {47}, (1998), no. 1, 297--309.



\bibitem{R1}{\sc A.\ Rashkovskii}, {\sl Analytic approximations of plurisubharmonic singularities},
 Math.\   Z.  275 (2013), 1217--1238. 



\bibitem{S} {\sc A.\ Sadullaev}, {\sl Plurisubharmonic measures and
capacities on complex manifolds}, Russian Math.\ Surveys 
36 (1981), 61--119.

\bibitem{T}{\sc P.\ Tworzewski}
{\sl Intersection theory in complex analytic geometry}
Ann.\ Polon.\ Math.\ {62} (1995), 177--191.


 



\end{thebibliography}
\end{document}